\RequirePackage{fix-cm}
\documentclass[smallextended,envcountsame]{svjour3} 
\smartqed  
\usepackage{graphicx, amssymb}
\usepackage{times,a4wide,mathrsfs}
\usepackage{amsmath}
\usepackage{amsfonts}

\newcommand{\C}{\mathbb{C}}

\newcommand{\QQ}{\mathbb{Q}}
\newcommand{\NN}{\mathbb{N}}
\newcommand{\PP}{\mathbb{P}}
\newcommand{\LLL}{\mathbb{L}}

\newcommand{\MM}{\mathcal M}

\newcommand{\rom}{\romannumeral}

\makeatletter
\newcommand*{\da@rightarrow}{\mathchar"0\hexnumber@\symAMSa 4B }
\newcommand*{\da@leftarrow}{\mathchar"0\hexnumber@\symAMSa 4C }
\newcommand*{\xdashrightarrow}[2][]{%
  \mathrel{%
    \mathpalette{\da@xarrow{#1}{#2}{}\da@rightarrow{\,}{}}{}%
  }%
}
\newcommand{\xdashleftarrow}[2][]{%
  \mathrel{%
    \mathpalette{\da@xarrow{#1}{#2}\da@leftarrow{}{}{\,}}{}%
  }%
}
\newcommand*{\da@xarrow}[7]{%
  \sbox0{$\ifx#7\scriptstyle\scriptscriptstyle\else\scriptstyle\fi#5#1#6\m@th$}%
  \sbox2{$\ifx#7\scriptstyle\scriptscriptstyle\else\scriptstyle\fi#5#2#6\m@th$}%
  \sbox4{$#7\dabar@\m@th$}%
  \dimen@=\wd0 %
  \ifdim\wd2 >\dimen@
    \dimen@=\wd2 %
  \fi
  \count@=2 %
  \def\da@bars{\dabar@\dabar@}%
  \@whiledim\count@\wd4<\dimen@\do{%
    \advance\count@\@ne
    \expandafter\def\expandafter\da@bars\expandafter{%
      \da@bars
      \dabar@ 
    }%
  }%
  \mathrel{#3}%
  \mathrel{%
    \mathop{\da@bars}\limits
    \ifx\\#1\\%
    \else
      _{\copy0}%
    \fi
    \ifx\\#2\\%
    \else
      ^{\copy2}%
    \fi
  }%
  \mathrel{#4}%
}
\makeatother

\newcommand\undermat[2]{
  \makebox[0pt][l]{$\smash{\underbrace{\phantom{
    \begin{matrix}#2\end{matrix}}}_{\text{$#1$}}}$}#2}

\newtheorem{convention}{Conventions}

\newtheorem{nonumbering}{Theorem}

\newtheorem{nonumberingc}{Corollary}

 \journalname{}

\begin{document}

\title{Some cubics with finite--dimensional motive}

\author{Robert Laterveer}

\institute{CNRS - IRMA, Universit\'e de Strasbourg \at
              7 rue Ren\'e Descartes \\
              67084 Strasbourg cedex\\
              France\\
              \email{laterv@math.unistra.fr}   }

\date{Received: date / Accepted: date}

\maketitle

\begin{abstract}
This small note presents in any dimension a family of cubics that have finite--dimensional motive (in the sense of Kimura). As an illustration, we verify a conjecture of Voevodsky for these cubics, and a conjecture of Murre for the Fano variety of lines of these cubics.
 \end{abstract}



\section{Introduction}

The notion of finite--dimensional motive, developed independently by Kimura and O'Sullivan \cite{Kim}, \cite{An}, \cite{MNP}, \cite{J4}, \cite{Iv} has given important new impetus to the study of algebraic cycles. To give but one example: thanks to this notion, we now know the Bloch conjecture is true for surfaces of geometric genus zero that are rationally dominated by a product of curves \cite{Kim}. It thus seems worthwhile to find concrete examples of varieties that have finite--dimensional motive, this being (at present) one of the sole means of arriving at a satisfactory understanding of Chow groups. 

The present note aims to contribute something to the list of examples of varieties with finite--dimensional motive, by considering cubic hypersurfaces.
In any dimension, there is one famous cubic known to have finite--dimensional motive: the Fermat cubic
  \[ (x_0)^3+(x_1)^3+\cdots+(x_{n+1})^3=0\ .\]
  The Fermat cubic has finite--dimensional motive because it is rationally dominated by a product of curves, and the indeterminacy locus is again of Fermat type \cite{Sh}. 
  In \cite{excubic4}, I proved finite--dimensionality for a certain $10$--dimensional family of cubic fourfolds.
  The main result of this note gives, for any dimension, a family of cubics (containing the Fermat cubic) with finite--dimensional motive:
  
  \begin{nonumbering}[=theorem \ref{main}] Let $X\subset\PP^{n+1}(\C)$ be a smooth cubic defined by an equation
  \[ f_0(x_0,\ldots,x_4)+f_1(x_5,\ldots,x_9) + \cdots + f_r(x_{5r},\ldots,x_{5r+4}) + f_{r+1}(x_{5(r+1)},\ldots,x_{n+1})       =0\ ,\]
  where $f_0, \ldots, f_r$ define smooth cubics of dimension $3$, and
  $f_{r+1}$ defines a smooth cubic of dimension $<3$ (i.e., $r=\lfloor{n+1\over 5}\rfloor$). Then $X$ has finite--dimensional motive.  
     \end{nonumbering}  

The proof is an elementary argument: employing the inductive structure exhibited by Shioda \cite{Sh}, \cite{KS}, one reduces theorem \ref{main} to the main result of \cite{excubic4}.

To illustrate how nicely the concept of finite--dimensionality allows to understand algebraic cycles, we provide an application to a conjecture of Voevodsky concerning smash--equivalence \cite{Voe}:

\begin{nonumberingc}[=corollary \ref{smash}] Let $X$ be a cubic as in theorem \ref{main}, and suppose $n=\dim X$ is odd. Then numerical equivalence and smash--equivalence coincide for all algebraic cycles on 
$X$.
\end{nonumberingc}
(For the definition of smash--equivalence, cf. definition \ref{sm}.)

We also establish the existence of a Chow--K\"unneth decomposition (as conjectured by Murre \cite{Mur}) for the Fano variety of lines on a cubic as in theorem \ref{main}; this is corollary \ref{CK}.

\vskip0.4cm

\begin{convention} All varieties will be projective irreducible varieties over $\C$.

For smooth $X$ of dimension $n$, we will denote by $A^j(X)=A_{n-j}(X)$ the Chow group $CH^j(X)\otimes{\QQ}$ of codimension $j$ cycles under rational equivalence. The notation $A^j_{num}(X)$ and $A^j_{\otimes}(X)$ will denote the subgroup of numerically trivial resp. smash-nilpotent cycles. The category $\MM_{\rm rat}$ will denote the (contravariant) category of Chow motives \cite{Sch}, \cite{MNP} over $\C$.
For a smooth projective variety, $h(X)=(X,\Delta_X,0)$ will denote its motive in $\MM_{\rm rat}$. 

\end{convention}

\section{Finite--dimensionality}

We refer to \cite{Kim}, \cite{An}, \cite{MNP}, \cite{Iv}, \cite{J4} for the definition of finite--dimensional motive. 
An essential property of varieties with finite--dimensional motive is embodied by the nilpotence theorem:

\begin{theorem}[Kimura \cite{Kim}]\label{nilp} Let $X$ be a smooth projective variety of dimension $n$ with finite--dimensional motive. Let $\Gamma\in A^n(X\times X)_{}$ be a correspondence which is numerically trivial. Then there is $N\in\NN$ such that
     \[ \Gamma^{\circ N}=0\ \ \ \ \in A^n(X\times X)_{}\ .\]
\end{theorem}

 Actually, the nilpotence property (for all powers of $X$) could serve as an alternative definition of finite--dimensional motive, as shown by a result of Jannsen \cite[Corollary 3.9]{J4}.
 
 \begin{conjecture}[Kimura \cite{Kim}]\label{findim}
  All smooth projective varieties have finite--dimensional motive.
  \end{conjecture}
  
   We are still far from knowing this, but at least there are quite a few non--trivial examples:
 
\begin{remark} 
The following varieties have finite--dimensional motive: abelian varieties, varieties dominated by products of curves \cite{Kim}, $K3$ surfaces with Picard number $19$ or $20$ \cite{P}, surfaces not of general type with $p_g=0$ \cite[Theorem 2.11]{GP}, certain surfaces of general type with $p_g=0$ \cite{GP}, \cite{PW}, \cite{V8}, Hilbert schemes of surfaces known to have finite--dimensional motive \cite{CM}, generalized Kummer varieties \cite[Remark 2.9(\rom2)]{Xu},
 threefolds with nef tangent bundle \cite{Iy} (an alternative proof is given in \cite[Example 3.16]{V3}), fourfolds with nef tangent bundle \cite{Iy2}, log--homogeneous varieties in the sense of \cite{Br} (this follows from \cite[Theorem 4.4]{Iy2}), certain threefolds of general type \cite[Section 8]{V5}, varieties of dimension $\le 3$ rationally dominated by products of curves \cite[Example 3.15]{V3}, varieties $X$ with $A^i_{AJ}(X)_{}=0$ for all $i$ \cite[Theorem 4]{V2}, products of varieties with finite--dimensional motive \cite{Kim}.
\end{remark}

\begin{remark}
It is an embarrassing fact that up till now, all examples of finite-dimensional motives happen to lie in the tensor subcategory generated by Chow motives of curves, i.e. they are ``motives of abelian type'' in the sense of \cite{V3}. On the other hand, there exist many motives that lie outside this subcategory, e.g. the motive of a very general quintic hypersurface in $\PP^3$ \cite[7.6]{Del}.
\end{remark}

\section{Main result}

\begin{theorem}\label{main} The following cubics have finite--dimensional motive (of abelian type):

\noindent
(\rom1) a smooth cubic $X\subset\PP^{5r+4}(\C)$ given by an equation
  \[  f_0(x_0,\ldots,x_4)+f_1(x_5,\ldots,x_9)+\cdots + f_r(x_{5r},\ldots,x_{5r+4})  =0\ ,\]
  where the $f_i$ define smooth cubics;
  
  \noindent
  (\rom2) a smooth cubic $X\subset\PP^{5r+5}(\C)$ given by an equation
   \[ f_0(x_0,\ldots,x_4)+f_1(x_5,\ldots,x_9)+\cdots + f_r(x_{5r},\ldots,x_{5r+4}) + (x_{5r+5})^3 =0\ ,\]   
     where the $f_i$ define smooth cubics;
   
  \noindent
  (\rom3) a smooth cubic $X\subset\PP^{5r+6}(\C)$ given by an equation 
   \[ f_0(x_0,\ldots,x_4)+f_1(x_5,\ldots,x_9)+\cdots + f_r(x_{5r},\ldots,x_{5r+4}) + f_{r+1}(x_{5r+5},x_{5r+6}) =0\ ,\]   
    where the $f_i$ define smooth cubics;

  \noindent
  (\rom4) a smooth cubic $X\subset\PP^{5r+2}(\C)$ given by an equation 
   \[  f_0(x_0,\ldots,x_4)+f_1(x_5,\ldots,x_9)+\cdots + f_{r-1}(x_{5r-5},\ldots,x_{5r-1}) + f_{r}(x_{5r},x_{5r+1},x_{5r+2}) =0\ ,\]   
    where the $f_i$ define smooth cubics;     
    
   \noindent
   (\rom5)  a smooth cubic $X\subset\PP^{5r+3}(\C)$ given by an equation 
   \[  f_0(x_0,\ldots,x_4)+f_1(x_5,\ldots,x_9)+\cdots + f_{r-1}(x_{5r-5},\ldots,x_{5r-1}) + f_{r}(x_{5r},\ldots,x_{5r+3}) =0\ ,\]   
    where the $f_i$ define smooth cubics;        
   
    \end{theorem}

  \begin{proof}
   The proof uses Shioda's inductive structure, in the guise of the following proposition (this is \cite[Remark 1.10]{KS}):
  
  \begin{proposition}[Katsura--Shioda \cite{KS}]\label{sh} Let $Z\subset \PP^{m_1+m_2}$ be a smooth hypersurface of degree $d$ defined by an equation
    \[ g_1(x_0,\ldots,x_{m_1})+g_2(x_{m_1+1},\ldots,x_{m_1+m_2})=0\ .\]
   Let $Z_1$ resp. $Z_2$ be the smooth hypersurfaces of dimension $m_1$ resp. $m_2-1$, defined as
     \[ g_1(x_0,\ldots,x_{m_1})+y^d=0\ ,\]
     resp.
     \[ g_2(x_{m_1+1},\ldots,x_{m_1+m_2})+z^d=0\ .\]   
     Then there exists a dominant rational map
     \[ \phi\colon\ \ Z_1\times Z_2\ \dashrightarrow\ Z\ ,\]
     and the indeterminacy of $\phi$ is resolved by blowing up the locus
     \[  \Bigl( Z_1\cap (y=0)\Bigr)\times \Bigl( Z_2\cap (z=0)\Bigr)\ \subset\ Z_1\times Z_2\ .\]
     \end{proposition}
   
 Let us first check theorem \ref{main} is true for $r=0$. In case (\rom1), this is clear since any cubic threefold has $A^\ast_{AJ}(X)=0$ and so has finite--dimensional motive. In case (\rom2), the fourfold $X$ has finite--dimensional motive thanks to \cite{excubic4}. In case (\rom3), we use that any cubic fivefold has $A^\ast_{AJ}(X)=0$ and so has finite--dimensional motive. In cases (\rom4) and (\rom5), the finite--dimensionality is again clear, since any curve and any del Pezzo surface has finite--dimensional motive.
 
 Next, let us suppose theorem \ref{main} is true for $r-1$, and let us prove this implies theorem \ref{main} for $r$.
 
 Let $X\subset \PP^{5r+4}$ be a cubic as in (\rom1). According to proposition \ref{sh}, there is a dominant rational map
   \[ \phi\colon\ \ X_1\times X_2\ \dashrightarrow\ X\ ,\]
   where 
   \[ \begin{split}  X_1 &= \bigl\{ f_0 + f_1 +\cdots +f_{r-1} + y ^3=0\bigr\}\ ,\\
                           X_2 &=\bigl\{ f_r + z^3=0\bigr\}\ .\\
                        \end{split}\]
           By induction, $X_1$ and $X_2$ have finite--dimensional motive. The indeterminacy of the rational map $\phi$ is resolved by blowing up  $Y_1\times Y_2\subset X_1\times X_2$, where $Y_1=X_1\cap (y=0)$ and $Y_2=X_2\cap (z=0)$. By induction, $Y_1$ and $Y_2$ have finite--dimensional motive, and so $X$ (being dominated by something with finite--dimensional motive) has finite--dimensional motive.
           
     Let $X\subset \PP^{5r+5}$ be a cubic as in (\rom2). According to proposition \ref{sh}, there is a dominant rational map
   \[ \phi\colon\ \ X_1\times X_2\ \dashrightarrow\ X\ ,\]
   where 
   \[ \begin{split}  X_1 &= \bigl\{ f_0 + f_1 +\cdots +f_{r-1} + y ^3=0\bigr\}\ ,\\
                           X_2 &=\bigl\{ f_r + (x_{5r+5})^3 + z^3=0\bigr\}\ .\\
                        \end{split}\]
           By induction, $X_1$ and $X_2$ have finite--dimensional motive. The indeterminacy of the rational map $\phi$ is resolved by blowing up  $Y_1\times Y_2\subset X_1\times X_2$, where $Y_1=X_1\cap (y=0)$ and $Y_2=X_2\cap (z=0)$. By induction, $Y_1$ and $Y_2$ have finite--dimensional motive, and so $X$ (being dominated by something with finite--dimensional motive) has finite--dimensional motive.
                         
    Let $X\subset \PP^{5r+6}$ be a cubic as in (\rom3). Applying proposition \ref{sh}, we find a dominant rational map
   \[ \phi\colon\ \ X_1\times X_2\ \dashrightarrow\ X\ ,\]
   where 
   \[ \begin{split}  X_1 &= \bigl\{ f_0 + f_1 +\cdots +f_{r} + y ^3=0\bigr\}\ ,\\
                           X_2 &=\bigl\{ f_{r+1}(x_{5r+5},x_{5r+6})+ z^3=0\bigr\}\ .\\
                        \end{split}\]
       The cubic  $X_1$ has finite--dimensional motive because we have just proven (\rom2) for $r$, and $X_2$ has finite--dimensional motive by the induction base. The indeterminacy of the rational map $\phi$ is resolved by blowing up  $Y_1\times Y_2\subset X_1\times X_2$, where $Y_1=X_1\cap (y=0)$ and $Y_2=X_2\cap (z=0)$. The cubics $Y_1$ and $Y_2$ have finite--dimensional motive (by (\rom1) for $r$ and the induction base), and so $X$ has finite--dimensional motive.
                         
 Next, let $X\subset \PP^{5r+2}$ be a cubic as in (\rom4). Applying proposition \ref{sh}, we find a dominant rational map
   \[ \phi\colon\ \ X_1\times X_2\ \dashrightarrow\ X\ ,\]
   where 
   \[ \begin{split}  X_1 &= \bigl\{ f_0 + f_1 +\cdots +f_{r-1} + y ^3=0\bigr\}\ ,\\
                           X_2 &=\bigl\{ f_{r}(x_{5r},x_{5r+1},x_{5r+2})+ z^3=0\bigr\}\ .\\
                        \end{split}\]
       The cubics  $X_1$ and $X_2$ have finite--dimensional motive by induction. The indeterminacy of the rational map $\phi$ is resolved by blowing up  $Y_1\times Y_2\subset X_1\times X_2$, where $Y_1=X_1\cap (y=0)$ and $Y_2=X_2\cap (z=0)$. The cubics $Y_1$ and $Y_2$ have finite--dimensional motive (by induction), 
       and so $X$ has finite--dimensional motive.
                         
 Finally, let $X\subset\PP^{5r+3}(\C)$ be as in (\rom5). There is a dominant rational map
     \[ \phi\colon\ \ X_1\times X_2\ \dashrightarrow\ X\ ,\]
   where 
   \[ \begin{split}  X_1 &= \bigl\{ f_0 + f_1 +\cdots +f_{r-1} + y ^3=0\bigr\}\ ,\\
                           X_2 &=\bigl\{ f_{r}(x_{5r},\ldots,x_{5r+3})+ z^3=0\bigr\}\ .\\
                        \end{split}\]
          The indeterminacy of $\phi$ is resolved by blowing up $Y_1\times Y_2\subset X_1\times X_2$, where $Y_1=X_1\cap (y=0)$ and $Y_2=X_2\cap (z=0)$. The varieties $X_1, X_2, Y_1, Y_2$ all have finite--dimensional motive, and so $X$ has finite--dimensional motive. This closes the proof.

          \end{proof}

 \begin{remark} 
 In \cite{moiq}, using a similar argument I prove that certain quartic hypersurfaces have finite--dimensional motive.
 
 We also mention, in passing, the following result concerning Chow groups of cubics: 
 Let $X\subset\PP^{n+1}(\C)$ be a smooth cubic of type
   \[   f_0(x_0,x_1,x_2) + f_1(x_3,x_4,x_5) + \cdots + f_r(x_{3r},x_{3r+1},x_{3r+2}) + f_{r+1}(x_{3r+3},\ldots,x_{n+1}) =0\ , \]
 where $r=\lfloor {n+1\over 3}\rfloor$. Then Colliot--Th\'el\`ene \cite{CT} has proven (exploiting the Shioda trick) that $X$ has universally trivial Chow group of $0$--cycles (i.e., there is an integral decomposition of the diagonal).

 \end{remark}

 \section{The Fano variety of lines}

 \begin{corollary} Let $X\subset\PP^{n+1}(\C)$ be a cubic as in theorem \ref{main}, and let $F(X)$ denote the Fano variety of lines on $X$. Then $F(X)$ has finite--dimensional motive (of abelian type).
 \end{corollary}
 
 \begin{proof} This follows from \cite{fanocubic}.
 \end{proof}
 
 \begin{corollary}\label{CK} Let $X\subset\PP^{n+1}(\C)$ be a cubic as in theorem \ref{main}, and let $F(X)$ denote the Fano variety of lines on $X$. Then there exists a Chow--K\"unneth decomposition for $F(X)$, i.e. a set of mutually orthogonal idempotents in $A^{2n-4}(F(X)\times F(X))$ summing to the diagonal and lifting the K\"unneth components.
 \end{corollary}
 
 \begin{proof} Thanks to the inclusion as a direct summand
    \[  h(F(X))\ \subset\ h(X^{[2]})(-2)\ \ \ \hbox{in}\ \MM_{\rm hom}\ \]
    (where $X^{[2]}$ is the Hilbert scheme)
    \cite{fanocubic}, 
    we know that $F(X)$ satisfies the Lefschetz standard conjecture (since $X^{[2]}$ does so). In particular, the K\"unneth components of the diagonal of $F(X)$ are algebraic \cite{K}. Nilpotence then allows to lift the K\"unneth components to a Chow--K\"unneth decomposition \cite[Lemma 3.1]{J2}.
  \end{proof}
  
 \begin{remark} The existence of a Chow--K\"unneth decomposition for {\em all\/} smooth projective varieties is conjectured by Murre \cite{Mur}.  
 For {\em any\/} cubic fourfold $X$, Shen and Vial \cite{SV} have explicitly constructed a Chow--K\"unneth decomposition for the Fano variety of lines $F(X)$. The argument is very different, since finite--dimensionality is not known for the Fano variety of a general cubic fourfold.
 \end{remark}

\section{Voevodsky's conjecture}

\begin{definition}[Voevodsky \cite{Voe}]\label{sm} Let $X$ be a smooth projective variety. A cycle $a\in A^r(X)$ is called {\em smash--nilpotent\/} 
if there exists $m\in\NN$ such that
  \[ \begin{array}[c]{ccc}  a^m:= &\undermat{(m\hbox{ times})}{a\times\cdots\times a}&=0\ \ \hbox{in}\  A^{mr}(X\times\cdots\times X)_{}\ .
  \end{array}\]
  \vskip0.6cm

Two cycles $a,a^\prime$ are called {\em smash--equivalent\/} if their difference $a-a^\prime$ is smash--nilpotent. We will write $A^r_\otimes(X)\subset A^r(X)$ for the subgroup of smash--nilpotent cycles.
\end{definition}

\begin{conjecture}[Voevodsky \cite{Voe}]\label{voe} Let $X$ be a smooth projective variety. Then
  \[  A^r_{num}(X)\ \subset\ A^r_\otimes(X)\ \ \ \hbox{for\ all\ }r\ .\]
  \end{conjecture}

\begin{remark} It is known \cite[Th\'eor\`eme 3.33]{An} that conjecture \ref{voe} implies (and is strictly stronger than) conjecture \ref{findim}. For partial results concerning conjecture \ref{voe}, cf. \cite{Seb2}, \cite{Seb}, \cite[Theorem 3.17]{V2}.
\end{remark}

As a corollary of finite--dimensionality, we can verify Voevodsky's conjecture for all odd--dimensional cubics as in theorem \ref{main}:

\begin{corollary}\label{smash} Let $X\subset\PP^{n+1}(\C)$ be a smooth cubic as in theorem \ref{main}. Assume $n$ is odd. Then
  \[  A^r_{num}(X)\ \subset\ A^r_\otimes(X)\ \ \ \hbox{for\ all\ }r\ .\]
  \end{corollary}
  
  \begin{proof} As $X$ is a hypersurface, the K\"unneth components $\pi_j$ are algebraic \cite{K} and the Chow motive of $X$ decomposes
    \[ h(X)=h_n(X)\oplus \bigoplus_j \LLL(n_j)\ \ \ \hbox{in}\ \MM_{\rm rat}\ .\]
   (Here $\LLL$ denotes the Lefschetz motive, and the motive $h_n(X)$ is defined simply as $(X,\Delta-\sum_{j\not=n}\pi_i,0)$.) 
    
    Since $A^r_{num}\bigl(\LLL(n_j)\bigr)=0$, we have
    \[ A^r_{num}(X)= A^r_{num}\bigl(h_n(X)\bigr)\ .\]
    The motive $h_n(X)$ is oddly finite--dimensional. (Indeed, since $n$ is odd we have that the motive $\hbox{Sym}^m h_n(X)\in \MM_{\rm hom}$ is $0$ for some $m>>0$. By finite--dimensionality, the same then holds in $\MM_{\rm rat}$.)
    The proposition now follows from the following result (which is \cite[Proposition 6.1]{Kim}, and which is also applied in \cite{KSeb} where I learned this):
    
    \begin{proposition}[Kimura \cite{Kim}] Suppose $M\in\MM_{\rm rat}$ is oddly finite--dimensional. Then
       \[ A^r_{}(M)\ \subset\ A^r_\otimes(M)\ \ \ \hbox{for\ all\ }r\ .\]
     \end{proposition}
     
     \end{proof}


\vskip0.6cm

\begin{acknowledgements} This note is a belated echo of the Strasbourg 2014---2015 groupe de travail based on the monograph \cite{Vo}. Thanks to all the participants for the pleasant and stimulating atmosphere. 
Many thanks to Yasuyo, Kai and Len for lots of enjoyable after--work ap\'eritifs. 
\end{acknowledgements}


\end{document}